\documentclass[10pt,reqno]{amsart}
\usepackage{latexsym,amsmath,amssymb,amscd}

\def\today{\ifcase \month \or
   January \or February \or March \or April \or
   May \or June \or July \or August \or
   September \or October \or November \or December \fi
   \space\number\day , \number\year}
   
\makeatletter
  \newcommand\@dotsep{4.5}
  \def\@tocline#1#2#3#4#5#6#7{\relax
     \ifnum #1>\c@tocdepth % then omit
     \else
     \par \addpenalty\@secpenalty\addvspace{#2}%
     \begingroup \hyphenpenalty\@M
     \@ifempty{#4}{%
     \@tempdima\csname r@tocindent\number#1\endcsname\relax
        }{%
         \@tempdima#4\relax
           }%
      \parindent\z@ \leftskip#3\relax \advance\leftskip\@tempdima\relax
      \rightskip\@pnumwidth plus1em \parfillskip-\@pnumwidth
       #5\leavevmode\hskip-\@tempdima #6\relax
       \leaders\hbox{$\m@th
       \mkern \@dotsep mu\hbox{.}\mkern \@dotsep mu$}\hfill
       \hbox to\@pnumwidth{\@tocpagenum{#7}}\par
       \nobreak
        \endgroup
         \fi}
\makeatother 
\begin{document}

%\DeclareRobustCommand{\SkipTocEntry}[4]{} 

\makeatletter
\@addtoreset{figure}{section}
\def\thefigure{\thesection.\@arabic\c@figure}
\def\fps@figure{h,t}
\@addtoreset{table}{bsection}

\def\thetable{\thesection.\@arabic\c@table}
\def\fps@table{h, t}
\@addtoreset{equation}{section}
\def\theequation{%\thesection.
\arabic{equation}}
\makeatother

\newcommand{\bfi}{\bfseries\itshape}

\newtheorem{theorem}{Theorem}
\newtheorem{acknowledgment}[theorem]{Acknowledgment}
\newtheorem{corollary}[theorem]{Corollary}
\newtheorem{definition}[theorem]{Definition}
\newtheorem{example}[theorem]{Example}
\newtheorem{lemma}[theorem]{Lemma}
\newtheorem{notation}[theorem]{Notation}
\newtheorem{proposition}[theorem]{Proposition}
\newtheorem{remark}[theorem]{Remark}
\newtheorem{setting}[theorem]{Setting}

\numberwithin{theorem}{section}
\numberwithin{equation}{section}

\newcommand{\1}{{\bf 1}}
\newcommand{\Ad}{{\rm Ad}}
\newcommand{\Alg}{{\rm Alg}\,}
\newcommand{\Aut}{{\rm Aut}\,}
\newcommand{\ad}{{\rm ad}}
\newcommand{\Borel}{{\rm Borel}}
\newcommand{\Ci}{{\mathcal C}^\infty}
\newcommand{\Cint}{{\mathcal C}^\infty_{\rm int}}
\newcommand{\Cpol}{{\mathcal C}^\infty_{\rm pol}}
\newcommand{\Der}{{\rm Der}\,}
\newcommand{\de}{{\rm d}}
\newcommand{\ee}{{\rm e}}
\newcommand{\End}{{\rm End}\,}
\newcommand{\ev}{{\rm ev}}
\newcommand{\hotimes}{\widehat{\otimes}}
\newcommand{\id}{{\rm id}}
\newcommand{\ie}{{\rm i}}
\newcommand{\iotaR}{\iota^{\rm R}}
\newcommand{\GL}{{\rm GL}}
\newcommand{\gl}{{{\mathfrak g}{\mathfrak l}}}
\newcommand{\Hom}{{\rm Hom}\,}
\newcommand{\Img}{{\rm Im}\,}
\newcommand{\Ind}{{\rm Ind}}
\newcommand{\Ker}{{\rm Ker}\,}
\newcommand{\Lie}{\text{\bf L}}
\newcommand{\Mt}{{{\mathcal M}_{\text t}}}
\newcommand{\m}{\text{\bf m}}
\newcommand{\pr}{{\rm pr}}
\newcommand{\Ran}{{\rm Ran}\,}
\renewcommand{\Re}{{\rm Re}\,}
\newcommand{\so}{\text{so}}
\newcommand{\spa}{{\rm span}\,}
\newcommand{\Tr}{{\rm Tr}\,}
\newcommand{\tw}{\ast_{\rm tw}}
\newcommand{\Op}{{\rm Op}}
\newcommand{\U}{{\rm U}}
\newcommand{\UCb}{{{\mathcal U}{\mathcal C}_b}}
\newcommand{\weak}{\text{weak}}

\newcommand{\QuadrHilb}{\textbf{QuadrHilb}}
\newcommand{\LieGr}{\textbf{LieGr}}

\newcommand{\CC}{{\mathbb C}}
\newcommand{\RR}{{\mathbb R}}
\newcommand{\TT}{{\mathbb T}}

\newcommand{\Ac}{{\mathcal A}}
\newcommand{\Bc}{{\mathcal B}}
\newcommand{\Cc}{{\mathcal C}}
\newcommand{\Dc}{{\mathcal D}}
\newcommand{\Ec}{{\mathcal E}}
\newcommand{\Fc}{{\mathcal F}}
\newcommand{\Hc}{{\mathcal H}}
\newcommand{\Jc}{{\mathcal J}}
\newcommand{\Lc}{{\mathcal L}}
\renewcommand{\Mc}{{\mathcal M}}
\newcommand{\Nc}{{\mathcal N}}
\newcommand{\Oc}{{\mathcal O}}
\newcommand{\Pc}{{\mathcal P}}
\newcommand{\Qc}{{\mathcal Q}}
\newcommand{\Sc}{{\mathcal S}}
\newcommand{\Tc}{{\mathcal T}}
\newcommand{\Vc}{{\mathcal V}}
\newcommand{\Uc}{{\mathcal U}}
\newcommand{\Xc}{{\mathcal X}}
\newcommand{\Yc}{{\mathcal Y}}
\newcommand{\Wig}{{\mathcal W}}

\newcommand{\Bg}{{\mathfrak B}}
\newcommand{\Fg}{{\mathfrak F}}
\newcommand{\Gg}{{\mathfrak G}}
\newcommand{\Ig}{{\mathfrak I}}
\newcommand{\Jg}{{\mathfrak J}}
\newcommand{\Lg}{{\mathfrak L}}
\newcommand{\Pg}{{\mathfrak P}}
\newcommand{\Sg}{{\mathfrak S}}
\newcommand{\Xg}{{\mathfrak X}}
\newcommand{\Yg}{{\mathfrak Y}}
\newcommand{\Zg}{{\mathfrak Z}}

\newcommand{\ag}{{\mathfrak a}}
\newcommand{\bg}{{\mathfrak b}}
\newcommand{\dg}{{\mathfrak d}}
\renewcommand{\gg}{{\mathfrak g}}
\newcommand{\hg}{{\mathfrak h}}
\newcommand{\kg}{{\mathfrak k}}
\newcommand{\mg}{{\mathfrak m}}
\newcommand{\n}{{\mathfrak n}}
\newcommand{\og}{{\mathfrak o}}
\newcommand{\pg}{{\mathfrak p}}
\newcommand{\sg}{{\mathfrak s}}
\newcommand{\tg}{{\mathfrak t}}
\newcommand{\ug}{{\mathfrak u}}
\newcommand{\zg}{{\mathfrak z}}

\newcommand{\ZZ}{\mathbb Z}
\newcommand{\NN}{\mathbb N}
\newcommand{\BB}{\mathbb B}
\newcommand{\HH}{\mathbb H}

\newcommand{\ep}{\varepsilon}

\newcommand{\hake}[1]{\langle #1 \rangle }

\newcommand{\scalar}[2]{\langle #1 ,#2 \rangle }
\newcommand{\vect}[2]{(#1_1 ,\ldots ,#1_{#2})}
\newcommand{\norm}[1]{\Vert #1 \Vert }
\newcommand{\normrum}[2]{{\norm {#1}}_{#2}}

\newcommand{\upp}[1]{^{(#1)}}
\newcommand{\p}{\partial}

\newcommand{\opn}{\operatorname}
\newcommand{\slim}{\operatornamewithlimits{s-lim\,}}
\newcommand{\sgn}{\operatorname{sgn}}

\newcommand{\seq}[2]{#1_1 ,\dots ,#1_{#2} }
\newcommand{\loc}{_{\opn{loc}}}

\makeatletter
\title[Algebras of symbols associated with the Weyl calculus]{Algebras of symbols associated with the Weyl calculus for Lie group representations}
\author{Ingrid Belti\c t\u a %$^{\ast)}$ 
 and Daniel Belti\c t\u a%$^{\ast\ast)}$
}
\address{Institute of Mathematics ``Simion Stoilow'' 
of the Romanian Academy, 
P.O. Box 1-764, Bucharest, Romania}
\email{Ingrid.Beltita@imar.ro}
\email{Daniel.Beltita@imar.ro}
%Lie group; semidirect product}
%\translator{}
%\dedicatory{}
%\thanks{\textit{File name}: \texttt{BB8\_August1&\_2010.tex}}
%\thanks{$^{\ast)}$, $^{\ast\ast)}$ Institute of Mathematics ``Simion Stoilow'' 
%of the Romanian Academy, 
%P.O. Box 1-764, Bucharest, Romania}
%\thanks{\textit{Email addresses:} Ingrid.Beltita@imar.ro, Daniel.Beltita@imar.ro}
\date{February 7, 2011}%{\today}
\makeatother

\begin{abstract}
We develop our earlier approach to the Weyl calculus for representations of infinite-dimensional Lie groups 
by establishing continuity properties of the Moyal product for symbols belonging to various modulation spaces. 
For instance, we prove that the modulation space of symbols $M^{\infty,1}$ is an associative Banach algebra and 
the corrresponding operators are bounded. 
We then apply the abstract results to two classes of representations, namely the unitary irreducible representations of nilpotent Lie groups, and the natural representations of the semidirect product groups that govern the magnetic Weyl calculus. 
The classical Weyl-H\"ormander calculus is obtained for 
the Schr\"odinger representations of the finite-dimensional Heisenberg groups, and in this case we recover 
the results obtained by J.~Sj\"ostrand in 1994. 
\end{abstract}

\maketitle

\tableofcontents

\section{Introduction}\label{Sect1}

A quite important class of symbols for pseudo-differential operators was introduced by 
J.~Sj\"ostrand in \cite{Sj94} (see also \cite{Sj95}). 
He denoted this class by $S(1)$ and pointed out that it has a number of remarkable properties, 
such as: 
\begin{enumerate}
\item\label{S1} 
For every symbol $a\in S(1)$ the corresponding operator $\Op (a)$ obtained by 
the pseudo-differential Weyl calculus is bounded on $L^2(\RR^n)$. 
\item\label{S2} 
The class $S(1)$ has a natural structure of unital involutive associative Banach algebra 
such that the mapping $\Op\colon S(1)\to\Bc(L^2(\RR^n))$ is a continuous $*$-homo\-morphism. 
\item\label{S3} 
If a symbol $a\in S(1)$ has the property that the operator $\Op(a)$ is invertible in $\Bc(L^2(\RR^n))$, 
then there exists $b\in S(1)$ such that $\Op(a)^{-1}=\Op(b)$. 
\end{enumerate}
It was later realized that the class $S(1)$ is actually the modulation space $M^{\infty,1}(\RR^{2n})$ 
(see for instance \cite{Gr01} for a broad discussion), and thus the above three properties become as many statements in representation theory of the Heisenberg groups. 

The aim of the present paper is to present the deep representation theoretic background of 
properties~\eqref{S1}--\eqref{S3}, in the sense that we obtain  below, in Theorem~\ref{th_sj}, their appropriate versions for some representations of infinite-dimensional Lie groups and their localized Weyl calculus proposed in our earlier papers \cite{BB09a} and \cite{BB10d}. 
We then apply these abstract results to two classes of representations: 

--- Representations of some infinite-dimensional Lie groups constructed as semidirect products of Lie groups and invariant function spaces thereon. 
We have pointed out in \cite{BB09a} and \cite{BB10a} that these semidirect products are the symmetry groups of the magnetic Weyl calculus developed for instance in \cite{MP04}, \cite{IMP07}, \cite{MP10} and \cite{IMP10}, 
and we thus find versions of the aforementioned properties in this setting. 

--- Unitary irreducible representations of arbitrary nilpotent Lie groups. 
The Weyl correspondence for these representations was developed by \cite{Pe94}, and we have later introduced in \cite{BB10c}  the modulation spaces in this framework and established continuity properties of the operators constructed by the corresponding Weyl calculus. 
In particular we found a space of symbols, which for the Heisenberg group reduces to Sj\"ostrand's class, and gives rise to bounded operators. 
We now show that that space of symbols has all the above properties~\eqref{S1}--\eqref{S3} in the case of an arbitrary irreducible representation of a nilpotent Lie group. 

The present paper is a sequel to \cite{BB10d} and relies on the methods developed there. 
In addition, we use some ideas contained in the deep analysis of Sj\"ostrand's class in \cite{Gr06}.  

\section{Sj\"ostrand's algebra of symbols in an abstract setting}\label{Sect2}

In this section we rely on some terminology and results of \cite{BB10d}. 
For the reader's convenience we recall in the first subsection (`Preliminaries') the framework and the basic ideas needed in the sequel. 
Sections \ref{Sect3} and \ref{Sect4} are devoted to discussing wide classes of examples satisfying the assumptions of this abstract framework. 

\subsection*{Preliminaries}
\begin{setting}\label{loc0}
\normalfont
Throughout this section we keep the following notation: 
\begin{enumerate}
\item $M$ is a locally convex Lie group (see \cite{Ne06}) with 
a smooth exponential mapping $\exp_M\colon\Lie(M)=\mg\to M$.
\item $\pi\colon M\to\Bc(\Hc)$ 
is a twice nuclearly smooth unitary representation (\cite[Def.~2.2]{BB10d}). 
For instance, $\pi$ may be any unitary irreducible representation of a nilpotent Lie group. 
See Section~\ref{Sect3} for another class of examples. 
\item $\Xi$ is a finite-dimensional vector space with a Lebesgue measure and $\Xi^*$ is a manifold with a polynomial structure (see \cite{Pe89}) and a Radon measure, endowed with a function 
$\langle\cdot,\cdot\rangle\colon\Xi^*\times\Xi\to\RR$ 
which is linear in the second variable and 
such that the ``Fourier transform''
$$\widehat{\cdot}
\colon L^1(\Xi)\to L^\infty(\Xi^*), \quad 
b(\cdot)\mapsto\widehat{b}(\cdot)
=\int\limits_{\Xi}\ee^{-\ie\langle\cdot,x\rangle}b(x)\,\de x $$
gives a linear topological isomorphism $\Sc(\Xi)\to\Sc(\Xi^*)$ and a 
unitary operator $L^2(\Xi)\to L^2(\Xi^*)$. 
The inverse of this transform is denoted by $a\mapsto\check a$. 
\item $\theta\colon \Xi\to\mg$ is a linear mapping such that 
\begin{enumerate}
\item $\pi$ satisfies the orthogonality relations along $\theta$ (\cite[Def.~3.2]{BB10d}); 
\item $\pi$ satisfies the density condition along $\theta$ (\cite[Def.~3.6]{BB10d}); 
\item the localized Weyl calculus for $\pi$ along $\theta$ is regular (\cite[Def.~3.10]{BB10d}); 
\item for every $u \in \U(\mg_{\CC})$ the function $\Vert \de \pi (u) \pi(\exp_M(\theta(\cdot)))\phi_0 \Vert$ has polynomial growth. 
\qed
\end{enumerate}
\end{enumerate}
\end{setting}

\begin{remark}\label{pol}
\normalfont
The setting of \cite{BB10d} is actually slightly narrower in the sense that $\Xi^*$ was supposed to be a finite-dimensional linear space 
and $\langle\cdot,\cdot\rangle\colon\Xi^*\times\Xi\to\RR$ was supposed to be a duality pairing. 
However, it is easily seen that the above setting ensures that the main results of \cite{BB10d} hold true. 
\qed
\end{remark}

\begin{notation}\label{not}
\normalfont
Here we summarize some additional notation related to the above setting: 
\begin{enumerate}
\item $\Hc_\infty$ is the space of smooth vectors for the representation $\pi$, 
which has a structure of nuclear Fr\'echet space. 
\item $\Hc_{-\infty}$ is the space of continuous antilinear functionals on $\Hc_\infty$, endowed with the topology of uniform convergence on the bounded sets. 
There exist the dense embeddings $\Hc_\infty\hookrightarrow\Hc\hookrightarrow\Hc_{-\infty}$,   
and the duality pairing $(\cdot\mid\cdot)\colon\Hc_{-\infty}\times\Hc_\infty\to\CC$ 
agrees with the scalar product of~$\Hc$.
\item $\Bc(\Hc)_\infty$ is the space of 
smooth vectors for the unitary representation 
$$\pi\otimes\bar\pi\colon M\times M\to\Bc(\Sg_2(\Hc)),\quad 
(\pi\otimes\bar\pi)(m_1,m_2)T=\pi(m_1)T\pi(m_2)^{-1}.$$
This is also a nuclear Fr\'echet space since  $\Bc(\Hc)_\infty\simeq\Hc_\infty\hotimes\Hc_\infty$ 
(see\cite[Eq.~(2.1)]{BB10d})). 
\item\label{not_ambiguity} 
$\Ac^{\pi,\theta}_\phi f\in\Cc(\Xi)\cap\Sc'(\Xi)$ 
is the \emph{ambiguity function} for the representation~$\pi$ along the mapping $\theta$, defined for  
$\phi\in\Hc_\infty$ and $f\in\Hc_{-\infty}$  
by  the formula
$$\Ac^{\pi,\theta}_\phi f\colon\Xi\to\CC,\quad 
(\Ac^{\pi,\theta}_\phi f)(\cdot)=(f\mid\pi(\exp_M(\theta(\cdot)))\phi). $$
\item $\Wig(f,\phi)\in\Sc'(\Xi^*)$ is the \emph{cross-Wigner distribution} 
for~$\pi$ along~$\theta$, defined by the formula 
$\widehat{\Wig(f,\phi)}=\Ac^{\pi,\theta}_\phi f$ for  
$\phi\in\Hc_\infty$ and $f\in\Hc_{-\infty}$. 
\item $M^{p,q}_\phi(\pi,\theta)$ are the \emph{modulation spaces} constructed for $p,q\in[1,\infty]$ with respect to a decomposition into a direct sum of subspaces $\Xi=\Xi_1\dotplus\Xi_2$ 
 and the \emph{window vector} $\phi\in\Hc_\infty\setminus\{0\}$. 
 Specifically, for any measurable function $F\colon \Xi\simeq\Xi_1\times\Xi_2\to\CC$ define 
$$\Vert F\Vert_{L^{p,q}(\Xi_1\times\Xi_2)}
=\Bigl(\int\limits_{\Xi_1}
\Bigl(\int\limits_{\Xi_2}
\vert F(X_1,X_2)\vert^p\de X_1 \Bigr)^{p/q}
\de X_2\Bigr)^{1/q}\in[0,\infty] $$
with the usual conventions if $p$ or $q$ is infinite. 
Then  
$$M^{p,q}_\Phi(\pi,\theta):=\{f\in\Hc_{-\infty}\mid\Vert f\Vert_{M^{p,q}_\phi(\pi,\theta)}
:=\Vert\Ac^{\pi,\theta}_\phi f \Vert_{L^{p,q}(\Xi_1\times\Xi_2)}<\infty\}.$$ 
If $1\le p_1\le p_2\le\infty$ and $1\le q_1\le q_2\le\infty$, 
then 
\begin{equation}\label{not_eq1}
M^{p_1,q_1}_\phi(\pi,\theta)\cap M^{\infty,\infty}_\phi(\pi,\theta)
\subseteq M^{p_2,q_2}_\phi(\pi,\theta)\cap M^{\infty,\infty}_\phi(\pi,\theta), 
\end{equation}
since  
$L^{p_1,q_1}(\Xi_1\times\Xi_2)\cap L^{\infty}(\Xi)\subseteq 
L^{p_2,q_2}(\Xi_1\times\Xi_2)\cap L^{\infty}(\Xi)$.  
\qed
\end{enumerate}
\end{notation}

\begin{remark}\label{moyal}
\normalfont 
We recall that the \emph{localized Weyl calculus for $\pi$ along~$\theta$} 
is the mapping $\Op^\theta\colon\widehat{L^1(\Xi)}\to\Bc(\Hc)$ given by 
\begin{equation}\label{moyal_eq1}
\Op^\theta(a)
=\int\limits_{\Xi} \check{a}(X)\pi(\exp_M(\theta(X)))\,\de X
\end{equation}
for $a\in\widehat{L^1(\Xi)}$, where we use weakly convergent integrals.
Under the present assumptions, it follows by \cite[Prop.~3.12]{BB10d} 
that the localized Weyl calculus defines a linear topological isomorphism 
$\Op^\theta\colon\Sc(\Xi^*)\to\Bc(\Hc)_\infty\simeq\Lc(\Hc_{-\infty},\Hc_\infty)$, 
its dual topological isomorphism $\Op^\theta\colon\Sc'(\Xi^*)\to\Lc(\Hc_\infty,\Hc_{-\infty})$,  
and the mapping 
\begin{equation}\label{moyal_eq2}
\Op^\theta\colon L^2(\Xi^*)\to\Sg_2(\Hc),
\end{equation} 
which is a unitary operator.  

If $a,b\in\Sc'(\Xi^*)$ and 
the operator product 
$\Op^\theta(a)\Op^\theta(b)\in\Lc(\Hc_\infty,\Hc_{-\infty})$ 
is well defined, then  
the \emph{Moyal product} $a\#^\theta b\in\Sc'(\Xi^*)$ 
is uniquely determined by the condition 
$$\Op^\theta(a\#^\theta b)=\Op^\theta(a)\Op^\theta(b).$$ 
Thus the Moyal product defines bilinear mappings 
$\Sc(\Xi^*)\times\Sc(\Xi^*)\to\Sc(\Xi^*)$ and 
$L^2(\Xi^*)\times L^2(\Xi^*)\to L^2(\Xi^*)$.
\qed
\end{remark}

\begin{definition}\label{basic}
\normalfont
Let us consider the semi-direct product $M\ltimes M$ defined by 
the action of $M$ on itself by inner automorphisms. 
Thus $M\ltimes M$ is a locally convex Lie group whose underlying manifold is $M\times M$ 
and the group operation is  
$$(m_1,m_2)(n_1,n_2)=(m_1n_1, n_1^{-1}m_2n_1n_2)$$
for all $m_1,m_2,n_1,n_2\in M$. 
There exists the natural continuous unitary representation 
$$\pi^\ltimes\colon M\ltimes M\to\Bc(\Sg_2(\Hc)),\quad 
\pi^\ltimes(m_1,m_2)T=\pi(m_1m_2)T\pi(m_1)^{-1}.$$
By using the unitary operator~\eqref{moyal_eq2}, 
we can construct the unitarily equivalent representation 
$\pi^\#\colon M\ltimes M\to\Bc(L^2(\Xi^*))$ 
(see \cite[Def.~3.13]{BB10d}). 
\qed
\end{definition}

\begin{remark}\label{ambiguity}
\normalfont 
The representation 
$\pi^\#\colon M\ltimes M\to\Bc(L^2(\Xi^*))$ has the following useful properties.
\begin{enumerate}
\item For every  $X_1,X_2\in\Xi$ and $f\in L^2(\Xi^*)$ we have \cite[Remark 3.14(2)]{BB10d}
\begin{equation}\label{ambiguity_eq1}
\pi^\#(\exp_{M\ltimes M}(\theta(X_1),\theta(X_2))f
=\ee^{\ie\langle\cdot,X_1+X_2\rangle}\#^\theta f\#^\theta \ee^{-\ie\langle\cdot,X_1\rangle}.
\end{equation}
\item For $F\in\Sc'(\Xi^*)$ and $\Phi\in\Sc(\Xi^*)$ we can consider the ambiguity function $\Ac^{\pi^\#,\theta\times\theta}_\Phi\colon\Xi\times\Xi\to\CC$ of 
the representation $\pi^\#$ along the linear mapping 
$\theta\times\theta\colon\Xi\times\Xi\to\mg\ltimes\mg$, 
just as in Notation~\ref{not}\eqref{not_ambiguity}, by 
the formula  
$$(\Ac^{\pi^\#,\theta\times\theta}_\Phi F)(X_1,X_2)=
(F\mid \pi^\#(\exp_{M\ltimes M}(\theta(X_1),\theta(X_2))\Phi))$$
for all $X_1,X_2\in\Xi$. 
We also define 
$$\Vert F\Vert_{M^{r,s}_\Phi(\pi^\#,\theta\times\theta)}
=\Bigl(\int\limits_{\Xi}
\Bigl(\int\limits_{\Xi}
\vert(\Ac^{\pi^\#,\theta\times\theta}_\Phi F)(X_1,X_2)\vert^r\de X_1 \Bigr)^{s/r}
\de X_2\Bigr)^{1/s}\in[0,\infty] $$
with the usual conventions if $r$ or $s$ is infinite. 
The space 
$$M^{r,s}_\Phi(\pi^\#,\theta\times\theta):=\{F\in\Sc'(\Xi^*)\mid\Vert F\Vert_{M^{r,s}_\Phi(\pi^\#,\theta\times\theta)}<\infty\}$$ 
is a \emph{modulation space} of symbols for the localized Weyl calculus $\Op^\theta$ associated with 
the unitary representation $\pi\colon M\to\Bc(\Hc)$ 
along with the linear mapping $\theta\colon \Xi\to\mg$
for the \emph{window vector} $\Phi\in\Sc(\Xi^*)\setminus\{0\}$. 
\end{enumerate}
\qed
\end{remark}

\subsection*{Ambiguity functions and matrix coefficients}

Here is a general version of the usual covariance property of the cross-Wigner distribution 
(see \cite[Prop.~4.3.2]{Gr01}).

\begin{lemma}\label{cov}
For all $f_1,f_2\in\Hc_{-\infty}$ and $X_1,X_2\in\Xi$ we have the equation 
$$\begin{aligned}
\Wig(\pi(\exp_M(\theta(X_1)))f_1, & \pi(\exp_M(\theta(X_2)))f_2) \\
&=\pi^\#(\exp_{M\ltimes M}(\theta(X_2),\theta(X_1-X_2)))\Wig(f_1,f_2)
\end{aligned}$$
in $\Sc'(\Xi^*)$. 
\end{lemma}

\begin{proof}
We have 
$$\begin{aligned}
\Op^\theta(\Wig(\pi(\exp_M(\theta(X_1)))f_1, &\pi(\exp_M(\theta(X_2)))f_2)) \\
&=(\cdot\mid\pi(\exp_M(\theta(X_2)))f_2)\pi(\exp_M(\theta(X_1)))f_1 \\
&=\pi(\exp_M(\theta(X_1)))((\cdot\mid f_2)f_1)\pi(\exp_M(\theta(-X_2))) \\
&=\Op^\theta(\ee^{\ie\langle\cdot,X_1\rangle})
\Op^\theta(\Wig(f_1,f_2))
\Op^\theta(\ee^{-\ie\langle\cdot,X_2\rangle}) \\
&=\Op^\theta(\ee^{\ie\langle\cdot,X_1\rangle}\#^\theta\Wig(f_1,f_2))
\#^\theta\ee^{-\ie\langle\cdot,X_2\rangle}) \\
&=\Op^\theta(\pi^\#(\exp_{M\ltimes M}(\theta(X_2),\theta(X_1-X_2)))\Wig(f_1,f_2)),
\end{aligned}$$
where the latter equality relies on~\eqref{ambiguity_eq1}. 

Then the assertion follows since $\Op^\theta\colon\Sc'(\Xi^*)\to\Lc(\Hc_\infty,\Hc_{-\infty})$ 
is a linear topological isomorphism. 
\end{proof}

\begin{theorem}\label{main}
If $\phi\in\Hc_\infty$ and $a\in\Sc'(\Xi^*)$, 
then for all $X_1,X_2\in\Xi$ we have 
$$(\Ac^{\pi^\#,\theta\times\theta}_{\Wig(\phi,\phi)}a)(X_1,X_2)
=(\Op^\theta(a)\phi_{X_1}\mid\phi_{X_1+X_2}) $$
where we denote $\phi_X:=\pi(\exp_M(\theta(X)))\phi\in\Hc_\infty$ 
for each $X\in\Xi$. 
\end{theorem}

\begin{proof}
First note that 
$$\begin{aligned}
\pi(\exp_M(\theta(-X_1-X_2)))
 \Op^\theta & (a)\pi(\exp_M(\theta(X_1))) \\
&=\Op^\theta(\ee^{-\ie\langle\cdot,X_1+X_2\rangle}\#^\theta a\#^\theta
\ee^{\ie\langle\cdot,X_1\rangle}) \\
&=\Op^\theta(\pi^\#(\exp_{M\ltimes M}(\theta(-X_1),\theta(-X_2)))a)
\end{aligned}$$
by \eqref{ambiguity_eq1}. 
Therefore 
\allowdisplaybreaks
\begin{align}
(\Op^\theta(a)\phi_{X_1}\mid\phi_{X_1+X_2}) 
&=(\pi(\exp_M(\theta(-X_1-X_2)))
 \Op^\theta(a)\pi(\exp_M(\theta(X_1)))\phi\mid\phi) \nonumber\\
&=(\Op^\theta(\pi^\#(\exp_{M\ltimes M}(\theta(-X_1),\theta(-X_2)))a)\phi  
\mid\phi) \nonumber\\
&= (\Op^\theta(\pi^\#(\exp_{M\ltimes M}(\theta(-X_1),\theta(-X_2)))a)\mid 
(\cdot\mid\phi)\phi) \nonumber\\
&= (\Op^\theta(\pi^\#(\exp_{M\ltimes M}(\theta(-X_1),\theta(-X_2)))a)\mid 
\Op^\theta(\Wig(\phi,\phi))) \nonumber\\
&= (\pi^\#(\exp_{M\ltimes M}(\theta(-X_1),\theta(-X_2)))a \mid 
\Wig(\phi,\phi)) \nonumber\\
&= (a \mid 
\pi^\#(\exp_{M\ltimes M}(\theta(X_1),\theta(X_2)))\Wig(\phi,\phi)) \nonumber\\
&= (\Ac^{\pi^\#,\theta\times\theta}_{\Wig(\phi,\phi)}a)(X_1,X_2), \nonumber
\end{align}
and this completes the proof. 
\end{proof}

\begin{corollary}\label{main_cor_gen}
Let $\phi\in\Hc_\infty$ and denote $\phi_X:=\pi(\exp_M(\theta(X)))\phi\in\Hc_\infty$ 
for each $X\in\Xi$. 
Also let $p,q\in[1,\infty]$.  
For $a\in\Sc'(\Xi^*)$ denote by $B_a$ the set of all measurable functions 
$\beta\colon\Xi\to[0,\infty]$ satisfying the condition 
\begin{equation}\label{main_cor_gen_eq1}
(\forall X\in\Xi)\quad \Vert(\Op^\theta(a)\phi_{\bullet}\mid\phi_{\bullet+X})\Vert_{L^p(\Xi)}
\le\beta(X). 
\end{equation}
Also define 
$$\beta_a\colon\Xi\to[0,\infty],\quad 
\beta_a(X)=\Vert(\Op^\theta(a)\phi_{\bullet}\mid\phi_{\bullet+X})\Vert_{L^p(\Xi)}. $$
Then we have 
$a\in M^{p,q}_{\Wig(\phi,\phi)}(\pi^\#,\theta\times\theta)$ 
if and only if $B_a\cap L^q(\Xi)\ne\emptyset$, and in this case 
$\beta_a\in B_a\cap L^q(\Xi)$ and 
\begin{equation}\label{main_cor_gen_eq2}
\Vert a\Vert_{M^{p,q}_{\Wig(\phi,\phi)}(\pi^\#,\theta\times\theta)}
=\inf_{\beta\in B_a\cap L^q(\Xi)}\Vert\beta\Vert_{L^q(\Xi)}
=\Vert\beta_a\Vert_{L^q(\Xi)} . 
\end{equation}
\end{corollary}

\begin{proof} 
If $a\in M^{p,q}_{\Wig(\phi,\phi)}(\pi^\#,\theta\times\theta)$, 
then the function $a_0\colon\Xi\to[0,\infty]$ defined by  $a_0(Y):=\Vert(\Ac^{\pi^\#,\theta\times\theta}_{\Wig(\phi,\phi)}a)(\cdot,Y)\Vert_{L^p(\Xi)}$ has the property $a_0\in L^q(\Xi)$ and moreover it follows at once by Theorem~\ref{main} that 
$\Vert(\Op^\theta(a)\phi_{\bullet}\mid\phi_{\bullet+X_2})\Vert_{L^p(\Xi)}\le a_0(X_2)$ for $X_2\in\Xi$. 
Hence condition~\eqref{main_cor_eq1} is satisfied for $\beta:=a_0$. 

Conversely, if \eqref{main_cor_gen_eq1} holds, then we get by Theorem~\ref{main} again that for all $Y\in\Xi$ we have $a_0(Y)\le\beta(Y)$, 
whence $\Vert a_0\Vert_{L^q(\Xi)}\le \Vert\beta\Vert_{L^q(\Xi)}<\infty$, 
and then $a\in M^{p,q}_{\Wig(\phi,\phi)}(\pi^\#,\theta\times\theta)$ 
and $\Vert a\Vert_{M^{p,q}_{\Wig(\phi,\phi)}(\pi^\#,\theta\times\theta)}\le 
\Vert\beta\Vert_{L^q(\Xi)}$. 

Equality~\eqref{main_cor_gen_eq2} is a by-product of the above reasoning, 
hence the proof is complete. 
\end{proof}

\begin{remark}\label{main_cor}
\normalfont 
By using Corollary~\ref{main_cor_gen} for $p=\infty$ we get the following abstract version of 
the \emph{almost diagonalization theorem} established in \cite[Th.~3.2]{Gr06}: 

Let $\phi\in\Hc_\infty$ and denote as above $\phi_X:=\pi(\exp_M(\theta(X)))\phi\in\Hc_\infty$ 
for each $X\in\Xi$. 
For $a\in\Sc'(\Xi^*)$ let $B_a$ be the set of all measurable functions 
$\beta\colon\Xi\to[0,\infty]$ satisfying the condition 
\begin{equation}\label{main_cor_eq1}
(\forall X_1,X_2\in\Xi)\quad \vert(\Op^\theta(a)\phi_{X_1}\mid\phi_{X_2})\vert\le\beta(X_1-X_2). 
\end{equation}
Also define 
$$\beta_a\colon\Xi\to[0,\infty],\quad 
\beta_a(X)=\sup_{Y\in\Xi}\vert(\Op^\theta(a)\phi_{X+Y}\mid\phi_X)\vert. $$
Then we have 
$a\in M^{\infty,q}_{\Wig(\phi,\phi)}(\pi^\#,\theta\times\theta)$ 
if and only if $B_a\cap L^q(\Xi)\ne\emptyset$, and in this case 
$\beta_a\in B_a\cap L^q(\Xi)$ and 
\begin{equation}\label{main_cor_eq2}
\Vert a\Vert_{M^{\infty,q}_{\Wig(\phi,\phi)}(\pi^\#,\theta\times\theta)}
=\inf_{\beta\in B_a\cap L^q(\Xi)}\Vert\beta\Vert_{L^p(\Xi)}
=\Vert\beta_a\Vert_{L^q(\Xi)} 
\end{equation}
whenever $1\le q\le\infty$.
\qed 
\end{remark}

\begin{definition}\label{ker1}
\normalfont
Let $\phi\in\Hc_\infty$ with $\Vert\phi\Vert=1$ and $\phi_X=\pi(\exp_M(\theta(X)))\phi\in\Hc_\infty$ 
for each $X\in\Xi$. 
For every $a\in\Sc'(\Xi^*)$ we define 
$$C_a\colon\Xi\times\Xi\to\CC,\quad C_a(X,Y):=(\Op^\theta(a)\phi_X\mid\phi_Y)$$
and the integral operator in $L^2(\Xi)$ defined by the integral kernel $C_a$ will be denoted by $$T_a\colon\Dc(T_a)\to L^2(\Xi).$$
Let us also denote by $V:=\Ac_\phi\colon\Hc\to L^2(\Xi)$ the isometry defined by the ambiguity functions. 
Note that the constant function $1\in\Sc'(\Xi^*)$ gives rise to the orthogonal projection 
$T_1=T_1^*=(T_1)^2\in\Bc(L^2(\Xi))$ with $\Ran T_1=\Ran V$ and $T_1T_a=T_aT_1=T_a$ for every $a\in\Sc'(\Xi)$. 
\qed 
\end{definition}

We now present the main idea which allows to use integral operators on $\Xi$ for the study of operators  $\Op^\theta (a)\colon\Hc_\infty\to\Hc_{-\infty}$,  $a\in \Sc'(\Xi^*)$. 

\begin{lemma}\label{ker2}
For arbitrary $a\in\Sc'(\Xi)$ we have 
$$\Op^\theta(a)= V^* T_a V\colon\Dc(\Op^\theta(a))\to\Hc$$
on the domain  
$$\Dc(\Op^\theta(a))=\{f\in\Hc\mid Vf\in\Dc(T_a)\}.$$
In particular, $T_a\in\Bc(L^2(\Xi))$ if and only if $\Op^\theta(a)\in\Bc(\Hc)$.
\end{lemma}

\begin{proof}
Use Definition~\ref{ker1} along with the fact that the representation 
$\pi\otimes\bar\pi\colon M\times M\to\Bc(\Sg_2(\Hc))$ satisfies the orthogonality relations 
along the linear mapping $\theta\times\theta\colon\Xi\times\Xi\to\mg\times\mg$ 
(see \cite[Lemma 3.8(1)]{BB10d}). 
\end{proof}

\begin{remark}\label{ker3}
\normalfont
If $a_1,a_2\in\Sc'(\Xi)$ and the operator product $T_{a_1}T_{a_2}$ is well defined in $L^2(\Xi)$,  so that $\Op^\theta(a_1)\Op^\theta(a_2)=V^*T_{a_1}T_{a_2}V\in\Lc(\Hc_\infty,\Hc_{-\infty})$ is well defined, then the Moyal product $a_1\#^\theta a_2\in\Sc'(\Xi^*)$ makes sense and we have 
\begin{equation}\label{ker_eq1}
C_{a_1\#^\theta a_2}(X,Z)=\int\limits_{\Xi}C_{a_1}(X,Y)C_{a_2}(Y,Z)\de Y
\end{equation}
for all $X,Z\in\Xi$. 
In fact, it follows by \cite[Lemma~3.19(4)]{BB10d} that for $X\in\Xi$ we have the integral 
$\Op^\theta(a_2)\phi_X
=\int\limits_\Xi(\Op^\theta(a_2)\phi_X\mid\phi_Y)\phi_Y\de Y$ 
convergent in $\Hc_{-\infty}$. 
Therefore, if $\Op^\theta(a_2)\phi_X\in\Dc(\Op^\theta(a_1))$, then 
$$\Op^\theta(a_1)\Op^\theta(a_2)\phi_X
=\int\limits_\Xi(\Op^\theta(a_2)\phi_X\mid\phi_Y)\Op^\theta(a_1)\phi_Y\de Y.$$ 
Hence we have 
$$(\Op^\theta(a_1)\Op^\theta(a_2)\phi_X\mid\phi_Z)
=\int\limits_\Xi(\Op^\theta(a_2)\phi_X\mid\phi_Y)(\Op^\theta(a_1)\phi_Y\mid\phi_Z)\de Y $$
for all $X,Z\in\Xi$.
\qed
\end{remark}

The next result is a generalization of the version of \cite[Prop.~0.1]{HTW07} without weights, 
which is recovered in the special case when $\pi$ is the Schr\"odinger representation of the Heisenberg group. 

\begin{corollary}\label{main_cor_mult}
Let $\phi\in\Hc_\infty$ and $p_1,p_2,p,q_1,q_2,q\in[1,\infty]$ 
such that 
$\frac{1}{p_1}+\frac{1}{p_2}=\frac{1}{p}$ and $\frac{1}{q_1}+\frac{1}{q_2}=1+\frac{1}{q}$. 
Then the Moyal product $\#^\theta$ defines a continuous bilinear map 
$$ M^{p_1,q_1}_{\Wig(\phi,\phi)}(\pi^\#,\theta\times\theta)
\times M^{p_2,q_2}_{\Wig(\phi,\phi)}(\pi^\#,\theta\times\theta)
\to M^{p,q}_{\Wig(\phi,\phi)}(\pi^\#,\theta\times\theta).$$
\end{corollary}

\begin{proof}  
For $j=1,2$ let $a_j\in M^{p_j,q_j}_{\Wig(\phi,\phi)}(\pi^\#,\theta\times\theta)$. 
We shall use the notation of Definition~\ref{ker1}.

Then there exists $\beta_j\in L^{q_j}(\Xi)$ such that 
$\Vert C_{a_j}(\cdot,\cdot+Y)\Vert_{L^{p_j}(\Xi)}\le \beta_j(Y)$ for every $Y\in\Xi$. 
On the other hand, it follows by \eqref{ker_eq1} that for $X_1,X_2\in\Xi$ we have 
$$\begin{aligned}
C_{a_1\#^\theta a_2}(X_1,X_1+X_2)
& =\int\limits_{\Xi}C_{a_1}(X_1,Y)C_{a_2}(Y,X_1+X_2)\de Y \\
& =\int\limits_{\Xi}C_{a_1}(X_1,X_1+Y)C_{a_2}(X_1+Y,X_1+X_2)\de Y
\end{aligned}$$
hence by Minkowski's inequality, and then H\"older's inequality, we get 
\allowdisplaybreaks
\begin{align}
\Vert C_{a_1\#^\theta a_2}(\cdot,\cdot+X_2)\Vert_{L^p(\Xi)}
& \le \int\limits_\Xi\Vert C_{a_1}(\cdot,\cdot+Y)C_{a_2}(\cdot+Y,\cdot+X_2)\Vert_{L^p(\Xi)}\de Y \nonumber\\
& \le \int\limits_\Xi\Vert C_{a_1}(\cdot,\cdot+Y)\Vert_{L^{p_1}(\Xi)} 
\Vert C_{a_2}(\cdot+Y,\cdot+X_2)\Vert_{L^{p_2}(\Xi)}\de Y \nonumber\\
& \le \int\limits_\Xi\Vert C_{a_1}(\cdot,\cdot+Y)\Vert_{L^{p_1}(\Xi)} 
\Vert C_{a_2}(\cdot,\cdot+X_2-Y)\Vert_{L^{p_2}(\Xi)}\de Y \nonumber\\
& \le \int\limits_\Xi\beta_1(Y)\beta_2(X_2-Y)\de Y \nonumber\\
&=:\beta (X_2) \nonumber
\end{align}
Since $\beta_j\in L^{q_j}(\Xi)$ for $j=1,2$, we have 
$\beta\in L^q(\Xi)$ and 
$$\Vert a_1\#^\theta a_2\Vert_{M^{p,q}_{\Wig(\phi,\phi)}(\pi^\#,\theta\times\theta)}
\le\Vert\beta\Vert_{L^q(\Xi)}\le 
\Vert\beta_1\Vert_{L^{q_1}(\Xi)}\Vert\beta_2\Vert_{L^{q_2}(\Xi)}.$$ 
By using Corollary~\ref{main_cor_gen} , 
it then follows that $a_1\#^\theta a_2\in M^{p,q}_{\Wig(\phi,\phi)}(\pi^\#,\theta\times\theta)$ 
and 
$$\Vert a_1\#^\theta a_2\Vert_{M^{p,q}_{\Wig(\phi,\phi)}(\pi^\#,\theta\times\theta)}\le 
\Vert a_1\Vert_{M^{p_1,q_1}_{\Wig(\phi,\phi)}(\pi^\#,\theta\times\theta)}
\Vert a_2\Vert_{M^{p_2,q_2}_{\Wig(\phi,\phi)}(\pi^\#,\theta\times\theta)},$$ 
which ends the proof. 
\end{proof}

\subsection*{The abstract version of Sj\"ostrand's algebra}

We can now prove the main result of the paper. 

\begin{theorem}\label{th_sj}
If $\phi\in\Hc_\infty$, then the following assertions hold: 
\begin{enumerate}
\item\label {th_sj_item1} 
For every $a\in M^{\infty,1}_{\Wig(\phi,\phi)}(\pi^\#,\theta\times\theta)$ 
we have $\Op^\theta(a)\in\Bc(\Hc)$ and moreover 
$\Vert\Op^\theta(a)\Vert\le 
\Vert a\Vert_{M^{\infty,1}_{\Wig(\phi,\phi)}(\pi^\#,\theta\times\theta)}$.
\item\label  {th_sj_item2} 
The Moyal product $\#^\theta$ makes the modulation space 
$M^{\infty,1}_{\Wig(\phi,\phi)}(\pi^\#,\theta\times\theta)$ 
into an involutive associative Banach algebra. 
\item\label{th_sj_item3} 
Let 
\begin{equation}\label{unit} 
\Mc^{\infty,1}_{\Wig(\phi,\phi)}(\pi^\#,\theta\times\theta)
=\CC 1+M^{\infty,1}_{\Wig(\phi,\phi)}(\pi^\#,\theta\times\theta).
\end{equation}
If $a_0\in \Mc^{\infty,1}_{\Wig(\phi,\phi)}(\pi^\#,\theta\times\theta)$ and the operator 
$\Op^\theta(a_0)$ is invertible in $\Bc(\Hc)$, 
then there exists $b_0\in \Mc^{\infty,1}_{\Wig(\phi,\phi)}(\pi^\#,\theta\times\theta)$ 
such that $\Op^\theta(a_0)^{-1}=\Op^\theta(b_0)$. 
\end{enumerate} 
\end{theorem}

\begin{proof}
To prove Assertion~\eqref {th_sj_item1}, 
let $a\in M^{\infty,1}_{\Wig(\phi,\phi)}(\pi^\#,\theta\times\theta)$ arbitrary 
and denote $\phi_X:=\pi(\exp_M(\theta(X)))\phi\in\Hc_\infty$ 
for each $X\in\Xi$. 
It follows by Corollary~\ref{main_cor} that there exists $\beta_a\in L^1(\Xi)$ 
such that 
$\Vert a\Vert_{M^{\infty,1}_{\Wig(\phi,\phi)}(\pi^\#,\theta\times\theta)}
=\Vert\beta_a\Vert_{L^1(\Xi)}$ and 
$$
(\forall X_1,X_2\in\Xi)\quad \vert(\Op^\theta(a)\phi_{X_1}\mid\phi_{X_2})\vert\le\beta_a(X_1-X_2). 
$$ 
Now let $f\in\Hc_\infty$ and recall from \cite[Lemma~3.19]{BB10d} that 
$f=\int\limits_\Xi(f\mid\phi_X)\phi_X\de X$, 
hence 
$$\begin{aligned}
\vert(\Op^\theta(a)f\mid\phi_Y)\vert
&\le\int\limits_\Xi\vert(f\mid\phi_X)\vert\cdot
\vert(\Op^\theta(a)\phi_X\mid\phi_Y)\vert\de X \\
&\le \int\limits_\Xi\vert(f\mid\phi_X)\vert\cdot
\beta_a(X-Y)\de X.
\end{aligned}$$
That is, $\vert(\Ac_\phi^{\pi,\theta}(\Op^\theta(a)f))(Y)
\le(\vert\Ac_\phi^{\pi,\theta}f\vert\ast\beta_a(-\cdot))(Y)$ for all 
$Y\in\Xi$. 
Therefore, 
$$\begin{aligned}
\Vert\Op^\theta(a)f\Vert
&=\Vert(\Ac_\phi^{\pi,\theta}(\Op^\theta(a)f))\Vert_{L^2(\Xi)}
\le\Vert \Ac_\phi^{\pi,\theta}f\Vert_{L^2(\Xi)} \Vert\beta_a\Vert_{L^1(\Xi)} \\
&=\Vert f\Vert\cdot\Vert a\Vert_{M^{\infty,1}_{\Wig(\phi,\phi)}(\pi^\#,\theta\times\theta)}. 
\end{aligned}$$
Since $f\in\Hc_\infty$ is arbitrary and $\Hc_\infty$ is dense in $\Hc$, the assertion follows. 

For Assertion~\eqref{th_sj_item2}, to see that $M^{\infty,1}_{\Wig(\phi,\phi)}(\pi^\#,\theta\times\theta)$ is closed under the Moyal product, just use Corollary~\ref{main_cor_mult} for $p_1=p_2=p$ and $q_1=q_2=q$. 
Next note that if 
$\Vert a\Vert_{M^{\infty,1}_{\Wig(\phi,\phi)}(\pi^\#,\theta\times\theta)}=0$, 
then $\Ac_{\Wig(\phi,\phi)}^{\pi^\#,\theta\times\theta}a=0$, and then 
it is straightforward to check that $a=0$. 
To prove that the norm of $M^{\infty,1}_{\Wig(\phi,\phi)}(\pi^\#,\theta\times\theta)$ is complete, it suffices to check that any Cauchy sequence $\{a_j\}_{j\ge1}$  has a convergent subsequence. 
By selecting a suitable subsequence, we may assume that 
$\Vert a_{j+1}-a_j\Vert_{M^{\infty,1}_{\Wig(\phi,\phi)}(\pi^\#,\theta\times\theta)}
<\frac{1}{2^j}$ for every $j\ge0$, where $a_0:=0$. 
It follows by Corollary~\ref{main_cor} that there exists 
$\beta_{j+1}\in L^1(\Xi)$ such that $\Vert\beta_{j+1}\Vert_{L^1(\Xi)}<\frac{1}{2^j}$ 
and 
\begin{equation}\label{th_sj_proof_eq1}
(\forall X_1,X_2\in\Xi)\quad \vert(\Op^\theta(a_{j+1}-a_j)\phi_{X_1}\mid\phi_{X_2})\vert
\le\beta_{j+1}(X_1-X_2).
\end{equation}
Note that $\beta:=\sum\limits_{j=1}^\infty\beta_j\in L^1(\Xi)$ and, by summing up the above inequalities for $j=0,\dots,k-1$ we get 
$$(\forall X_1,X_2\in\Xi)\quad \vert(\Op^\theta(a_k)\phi_{X_1}\mid\phi_{X_2})\vert
\le(\beta_1+\cdots+\beta_k)(X_1-X_2)\le\beta(X_1-X_2).$$
On the other hand, since $\{a_k\}_{k\ge1}$ is a Cauchy sequence 
in $M^{\infty,1}_{\Wig(\phi,\phi)}(\pi^\#,\theta\times\theta)$, 
it follows by Assertion~\eqref {th_sj_item1} that there exists 
an operator $T\in\Bc(\Hc)$ such that $\lim\limits_{k\to\infty}\Vert\Op^\theta(a_k)-T\Vert=0$. 
It follows by the above inequalities for $k\to\infty$ that 
$$(\forall X_1,X_2\in\Xi)\quad 
\vert(T\phi_{X_1}\mid\phi_{X_2})\vert
\le\beta(X_1-X_2).$$
Moreover, it follows by Remark~\ref{moyal} that $T=\Op^\theta(a)$ for some $a\in\Sc'(\Xi^*)$, and then $a\in M^{\infty,1}_{\Wig(\phi,\phi)}(\pi^\#,\theta\times\theta)$ by the above inequality along with Corollary~\ref{main_cor}. 
Finally, by summing up the inequalities \eqref{th_sj_proof_eq1} 
for $j=k,k+1,\dots$ and using Corollary~\ref{main_cor} again, we get 
$\Vert a-a_k\Vert_{M^{\infty,1}_{\Wig(\phi,\phi)}(\pi^\#,\theta\times\theta)}
\le \sum\limits_{j=k}^\infty\frac{1}{2^j}=\frac{1}{2^{k-1}}$ for arbitrary $k\ge1$, hence $a=\lim\limits_{k\to\infty}a_k$ in $M^{\infty,1}_{\Wig(\phi,\phi)}(\pi^\#,\theta\times\theta)$. 

For Assertion~\eqref{th_sj_item3} let 
$a_0\in \Mc^{\infty,1}_{\Wig(\phi,\phi)}(\pi^\#,\theta\times\theta)$ and  assume that the operator 
$\Op^\theta(a_0)$ is invertible in $\Bc(\Hc)$. 
There exist $\alpha\in\CC$ and $a_{00}\in M^{\infty,1}_{\Wig(\phi,\phi)}(\pi^\#,\theta\times\theta)$ 
such that $a_0=\alpha+a_{00}$. 
We shall use the notation of Definition~\ref{ker1} and also  
recall that for the symbol $1\in \Mc^{\infty,1}_{\Wig(\phi,\phi)}(\pi^\#,\theta\times\theta)$ 
we get the operator $T_1\in\Bc(L^2(\Xi))$ with the properties $T_1=T_1^*=(T_1)^2$ and $\Ran T_1=\Ran V$. 
Moreover, for every 
$a\in \Mc^{\infty,1}_{\Wig(\phi,\phi)}(\pi^\#,\theta\times\theta)$  we have 
$T_aT_1=T_1T_a=T_a$, and in particular $T_a$ vanishes on $(\Ran T_1)^\perp$. 

It then follows that if $z\in\CC\setminus\{\alpha\}$ and the operator 
$$z\1-\Op^\theta(a_0)=\Op^\theta(z-a_0)=\Op^\theta(z-\alpha-a_{00})$$ 
is invertible in $\Bc(\Hc)$, 
then $(z-\alpha)(\1-T_1)+T_{z-\alpha-a_{00}}=(z-\alpha)\1-T_{a_{00}}$ is invertible in $\Bc(L^2(\Xi))$. 
On the other hand, since $a_{00}\in M^{\infty,1}_{\Wig(\phi,\phi)}(\pi^\#,\theta\times\theta)$, 
it follows by Remark~\ref{main_cor} that there exists $\beta_0\in L^1(\Xi)$ such that 
the integral kernel $C_{a_{00}}$ of $T_{a_{00}}$ satisfies the estimate 
$\vert C_{a_{00}}(X-Y)\vert\le\beta_0(X-Y)$ for all $X,Y\in\Xi$. 
We then get by \cite[Th.~5.4.7]{Ku99} (see also \cite{Ku01}) that 
$((z-\alpha)\1-T_{a_{00}})^{-1}=(z-\alpha)^{-1}\1-N_z$, 
where $N_z\in\Bc(L^2(\Xi))$ is an integral operator whose 
kernel $K_{N_z}$ satisfies a similar estimate 
$\vert K_{N_z}(X-Y)\vert\le\beta_z(X-Y)$ for all $X,Y\in\Xi$ 
and a suitable function $\beta_z\in L^1(\Xi)$. 
Since $T_{a_{00}}T_1=T_1T_{a_{00}}=T_{a_{00}}$, it follows that $N_zT_1=T_1N_z=N_z$.  
By using the fact that $\Op^\theta\colon\Sc'(\Xi^*)\to\Lc(\Hc_\infty,\Hc_{-\infty})$ is a 
linear isomorphism (see \cite[Rem.~3.11]{BB10d}) and Lemma~\ref{ker2}, 
we then get $b_z\in\Sc'(\Xi^*)$ such that $\Op^\theta(b_z)\in\Bc(\Hc)$ 
and $T_{b_z}=N_z$. 
Moreover, the estimates satisfied by the integral kernel of $N_z$ show that 
actually $b_z\in M^{\infty,1}_{\Wig(\phi,\phi)}(\pi^\#,\theta\times\theta)$ by Remark~\ref{main_cor} again. 
 
 We have thus shown that if $z\in\CC\setminus\{\alpha\}$ and $z\1-\Op^\theta(a_0)$ (which is equal to $\Op^\theta(z-a_0)$) is invertible in $\Bc(\Hc)$, then there exists $b_z\in M^{\infty,1}_{\Wig(\phi,\phi)}(\pi^\#,\theta\times\theta)$ such that $\Op^\theta(z-a_0)^{-1}=\Op^\theta((z-\alpha)^{-1}-b_z)$. 
 Thus we can see that $z-a_0$ is invertible in the unital Banach algebra $\Mc^{\infty,1}_{\Wig(\phi,\phi)}(\pi^\#,\theta\times\theta)$ and its inverse is 
 $(z-\alpha)^{-1}-b_z$. 
 In particular, $z\mapsto b_z$ is a holomorphic mapping from 
 the complement of the spectrum of $\Op^\theta(a_0)$ into $\Mc^{\infty,1}_{\Wig(\phi,\phi)}(\pi^\#,\theta\times\theta)$. 
Now, since $\Op^\theta(a_0)\in\Bc(\Hc)$ is an invertible operator, there exists a piecewise smooth closed curve that does not contain $\alpha$ and surrounds the spectrum of $\Op^\theta(a_0)$, 
and we have by holomorphic functional calculus 
$$\Op^\theta(a_0)^{-1}
=\frac{1}{2\pi\ie}\int\limits_\gamma\frac{1}{z}(z-\Op^\theta(a_0))^{-1}\de z.$$
Since $\Mc^{\infty,1}_{\Wig(\phi,\phi)}(\pi^\#,\theta\times\theta)$ is a unital Banach algebra, we can define 
$$b_0:=\frac{1}{2\pi\ie}\int\limits_\gamma\frac{1}{z}((z-\alpha)^{-1}-b_z)\de z
\in \Mc^{\infty,1}_{\Wig(\phi,\phi)}(\pi^\#,\theta\times\theta).$$ 
Then 
$$\begin{aligned}
\Op^\theta(b_0)
&= \frac{1}{2\pi\ie}\int\limits_\gamma\frac{1}{z}\Op^\theta((z-\alpha)^{-1}-b_z)\de z
=\frac{1}{2\pi\ie}\int\limits_\gamma\frac{1}{z}(z-\Op^\theta(a_0))^{-1}\de z \\
&=\Op^\theta(a_0)^{-1},
\end{aligned}$$ 
which completes the proof. 
\end{proof}

\begin{remark}
\normalfont
A more general result on the continuity of the operators $\Op^\theta(a)$ 
on modulation spaces was obtained in \cite{BB10d} by a completely different method based on continuity properties of the cross-Wigner distribution. 
\qed
\end{remark}

\section{Applications to the magnetic Weyl calculus}\label{Sect3}

\begin{notation}
\normalfont
Let $G$ be a simply connected, nilpotent Lie group with the Lie algebra $\gg$ 
and the inverse of the exponential map denoted by $\log_G\colon G\to\gg$. 
We denote by 
$\lambda\colon G\to\End(\Ci(G))$, $g\mapsto\lambda_g$,  
the left regular representation defined by $(\lambda_g\phi)(x)=\phi(g^{-1}x)$ 
for every $x,g\in G$ and $\phi\in\Ci(G)$. 
Moreover, we denote by $\1$ the constant function which is identically equal to~1 on $G$. 
(This should not be confused with the unit element of $G$, which is denoted in the same way.)

If the 
space of globally defined smooth vector fields on $G$ 
(that is, global sections in its tangent bundle) is  denoted by $\Xg(G)$ 
and the space of globally defined smooth 1-forms  
(that is, global sections in its cotangent bundle) is denoted by $\Omega^1(G)$,  
then there exists a natural bilinear map 
$$\langle\cdot,\cdot\rangle\colon\Omega^1(G)\times\Xg(G)\to\Ci(G)$$
defined as usually by evaluations at every point of~$G$. 

For arbitrary $g\in G$ we denote the corresponding right-translation mapping by 
$R_g\colon G\to G$, $h\mapsto hg$. 
Then we define the injective linear mapping 
$$\iotaR\colon\gg\to\Xg(G)$$ 
by $(\iotaR X)(g)=(T_{\1}(R_g))X\in T_g G$ 
for all $g\in G$ and $X\in\gg$. 

Moreover, we define 
$$\Xi=\Xi^*:=\gg\times\gg^*$$ and 
the symplectic duality pairing 
$$\langle\cdot,\cdot\rangle\colon\Xi^*\times\Xi\to\RR,\quad 
((X_1,\xi_1),(X_2,\xi_2))\mapsto
\langle\xi_1,X_2\rangle-\langle\xi_2,X_1\rangle$$
where $\langle\cdot,\cdot\rangle\colon\gg^*\times\gg\to\RR$ is the natural duality pairing. 
\qed
\end{notation}

\begin{setting}\label{orbit0}
\normalfont  
Throughout this section we denote by $\Fc$ a linear space of real functions on the Lie group $G$ which is endowed with a sequentially complete, locally convex topology and satisfies the following conditions: 
\begin{enumerate}
\item\label{orbit0_item1}
The linear space $\Fc$ is invariant under the representation of $G$ by left translations, 
that is, if $\phi\in\Fc$ and $g\in G$ then $\lambda_g\phi\in\Fc$. 
\item\label{orbit0_item2}
There exist the continuous inclusion maps 
$\gg^*\hookrightarrow\Fc\hookrightarrow\Cpol(G)$, 
where the embedding $\gg^*\hookrightarrow\Fc$ is given by $\xi\mapsto\xi\circ\log_G$. 
\item\label{orbit0_item3}
The mapping $G\times\Fc\to\Fc$, $(g,\phi)\mapsto\lambda_g\phi$ is smooth. 
For every $\phi\in\Fc$ we denote by $\dot{\lambda}(\cdot)\phi\colon\gg\to\Fc$ 
the differential of the mapping $g\mapsto\lambda_g\phi$ at the point $\1\in G$. 
\end{enumerate}
For instance, the function space $\Cpol(G)$ is admissible. Here  $\Cpol(G)$ is the space of smooth functions $\phi\colon G \to \RR$ such that the function $\phi\circ\log_G\colon \gg \to \RR$ and its partial derivatives have polynomial growth. 
\qed
\end{setting}

\begin{definition}
\normalfont
We define the semidirect product $M=\Fc\rtimes_\lambda G$, which is a locally convex Lie group, and the unitary representation 
$$\pi\colon M\to\Bc(L^2(G)), \quad \pi(\phi,g)f=\ee^{\ie\phi}\lambda_g f 
\text{ for }\phi\in\Fc, g\in G,\text{ and }f\in L^2(G).$$
If we have $A\in\Omega^1(G)$ with \emph{$\Fc$-growth}, in the sense that   
$\langle A,\iotaR X\rangle\in\Fc$ whenever $X\in\gg$,  
then we define the linear mapping 
$$\theta^A\colon\gg\times\gg^*\to\mg=\Fc\ltimes_{\dot\lambda}\gg,\quad 
(X,\xi)\mapsto(\xi\circ\log_G+\langle A,\iotaR X\rangle,X).$$
\qed
\end{definition}

\begin{remark}\label{conditions}
\normalfont
 The representation $\pi$ is twice nuclearly smooth and its space of smooth vectors 
is the Schwartz space $\Sc(G)$ (\cite[Cor.4.5]{BB10d}) 
and the following assertions hold for every 1-form  
$A\in\Omega^1(G)$ with $\Fc$-growth: 
\begin{enumerate}
\item\label{conditions_item1} 
The representation $\pi$ satisfies the orthogonality relations along the mapping~$\theta^A$. 
\item\label{conditions_item2} 
The representation $\pi$ satisfies the density condition  along~$\theta^A$.
\item\label{conditions_item3} 
The localized Weyl calculus for $\pi$ along $\theta^A$ is regular and defines a unitary operator 
$\Op^{\theta^A}\colon L^2(\gg\times\gg^*)\to\Sg_2(L^2(G))$.
\item\label{conditions_item4} 
If $u\in\U(\mg_{\CC})$ and $\phi\in\Sc(G)$,  
the function $\Vert\de\pi(\Ad_{\U(\mg_{\CC})}(\exp_M(\theta^A(\cdot)))u)\phi\Vert$ has polynomial growth on~$\gg\times\gg^*$.
\end{enumerate} 
These properties have been established in \cite[Cor.4.5]{BB10d}, 
and it thus follows that all of the conditions of Setting~\ref{loc0} are satisfied in the present setting provided by the representation $\pi$ and the linear mapping $\theta^A$. 

Just as in \cite{BB09a} we shall denote the corresponding Moyal product $\#^{\theta^A}$ simply by $\theta^A$ and localized Weyl calculus for $\pi$ along $\theta^A$ is denoted by $\Op^A(\cdot)$ and is called the \emph{magnetic Weyl calculus} associated with the magnetic potential $A\in\Omega^1(G)$. 
The corresponding \emph{magnetic field} is $B:=\de A\in\Omega^2(G)$. 
\qed
\end{remark}

\begin{theorem}\label{sj_mag}
If $\phi\in\Sc(G)$, then the following assertions hold: 
\begin{enumerate}
\item\label {sj_mag_item1} 
For every $a\in M^{\infty,1}_{\Wig(\phi,\phi)}(\pi^\#,\theta^A\times\theta^A)$ 
we have $\Op^A(a)\in\Bc(L^2(G))$ and moreover 
$\Vert\Op^A(a)\Vert\le 
\Vert a\Vert_{M^{\infty,1}_{\Wig(\phi,\phi)}(\pi^\#,\theta^A\times\theta^A)}$.
\item\label  {sj_mag_item2} 
The Moyal product $\#^A$ makes the modulation space 
$M^{\infty,1}_{\Wig(\phi,\phi)}(\pi^\#,\theta^A\times\theta^A)$ 
into an associative Banach algebra. 
\item\label{sj_mag_item3} 
If $a_0\in \Mc^{\infty,1}_{\Wig(\phi,\phi)}(\pi^\#,\theta^A\times\theta^A)$ and 
$\Op^A(a_0)\in\Bc(L^2(G))$ is invertible, 
then there exists $b_0\in \Mc^{\infty,1}_{\Wig(\phi,\phi)}(\pi^\#,\theta^A\times\theta^A)$ 
such that $\Op^A(a_0)^{-1}=\Op^A(b_0)$. 
\end{enumerate} 
\end{theorem}

\begin{proof}
The above Remark~\ref{conditions} shows that 
Corollary~\ref{main_cor_mult} and Theorem~\ref{th_sj} apply, 
and then the assertions follow.
\end{proof}

\section{Applications to representations of nilpotent Lie groups}\label{Sect4}

All of the conditions of Setting~\ref{loc0} are satisfied if $M$ is a finite-dimensional nilpotent Lie group, $\pi$ is a unitary irreducible representation with the corresponding coadjoint orbit $\Xi^*$, $\Xi$ is a predual of the coadjoint orbit $\Xi$ (in the sense of \cite{BB10c}) and $\theta\colon\Xi\hookrightarrow\mg$ is the embedding map. 
This will be the setting of the present section, and our point here is to describe how the abstract results of Section~\ref{Sect2} can be specialized in this framework, and also to point out how they can be further sharpened in the special case when $\pi$ is a square-integrable representation modulo the center. 

\begin{setting}\label{predual_sett}
\normalfont 
Throughout this section we shall use the following notation:
\begin{enumerate}
 \item Let $G$ be a connected, simply connected, nilpotent Lie group with Lie algebra~$\gg$.  
 Then the exponential map $\exp_G\colon\gg\to G$ is a diffeomorphism 
 with the inverse denoted by $\log_G\colon G\to\gg$. 
 \item We denote by $\gg^*$ the linear dual space to $\gg$ and 
  by $\hake{\cdot,\cdot}\colon\gg^*\times\gg\to{\RR}$ the natural duality pairing. 
 \item Let $\xi_0\in\gg^*$ with the corresponding coadjoint orbit $\Oc:=\Ad_G^*(G)\xi_0\subseteq\gg^*$.  
 \item Let $\pi\colon G\to\Bc(\Hc)$ be any unitary irreducible representations 
associated with the coadjoint orbit $\Oc$ by Kirillov's theorem (\cite{Ki62}).
 \item The {\it isotropy group} at $\xi_0$ is $G_{\xi_0}:=\{g\in G\mid\Ad_G^*(g)\xi_0=\xi_0\}$ 
 with the corresponding {\it isotropy Lie algebra} $\gg_{\xi_0}=\{X\in\gg\mid\xi_0\circ\ad_{\gg}X=0\}$. 
 If we denote the {\it center} of $\gg$ by $\zg:=\{X\in\gg\mid[X,\gg]=\{0\}\}$, 
 then $\zg\subseteq\gg_{\xi_0}$. 
 \item Let $n:=\dim\gg$ and fix a sequence of ideals in $\gg$, 
$$\{0\}=\gg_0\subset\gg_1\subset\cdots\subset\gg_n=\gg$$
such that $\dim(\gg_j/\gg_{j-1})=1$ and $[\gg,\gg_j]\subseteq\gg_{j-1}$ 
for $j=1,\dots,n$. 
 \item Pick any $X_j\in\gg_j\setminus\gg_{j-1}$ for $j=1,\dots,n$, 
so that the set $\{X_1,\dots,X_n\}$ will be a {\it Jordan-H\"older basis} in~$\gg$. 
\end{enumerate} 
Also consider the set of \emph{jump indices} of the coadjoint orbit $\Oc$ 
with respect to the aforementioned Jordan-H\"older basis, 
$$e:=\{j\in\{1,\dots,n\}\mid \gg_j\not\subseteq\gg_{j-1}+\gg_{\xi_0}\}
=\{j\in\{1,\dots,n\}\mid X_j\not\in\gg_{j-1}+\gg_{\xi_0}\}$$ 
and then define the corresponding \emph{predual of the coadjoint orbit}~$\Oc$, 
$$\gg_e:=\spa\{X_j\mid j\in e\}\subseteq\gg.$$
We note the direct sum decomposition $\gg=\gg_{\xi_0}\dotplus\gg_e$. 
\qed
\end{setting}

\begin{theorem}\label{sq_part}
If the representation $\pi$ is square integrable modulo the center, then the following assertions hold: 
\begin{enumerate}
\item If $p_1,p_2,p,q_1,q_2,q\in[1,\infty]$ 
satisfy the conditions 
$\frac{1}{p_1}+\frac{1}{p_2}=\frac{1}{p}$ and $\frac{1}{q_1}+\frac{1}{q_2}=1+\frac{1}{q}$,  
then the Moyal product $\#^\theta$ defines a continuous bilinear map 
$$ M^{p_1,q_1}(\pi^\#)\times M^{p_2,q_2}(\pi^\#)\to M^{p,q}(\pi^\#).$$
\item The Moyal product $\#^\theta$ makes the modulation space 
$M^{\infty,1}(\pi^\#)$ 
into an associative involutive Banach algebra and the Weyl calculus defines an injective continuous $*$-homomorphism 
$\Op\colon M^{\infty,1}(\pi^\#)\to\Bc(\Hc)$. 
\item
If $a_0\in \Mc^{\infty,1}(\pi^\#)$ and  
$\Op^\theta(a_0)\in\Bc(\Hc)$ is an invertible operator, 
then there exists $b_0\in \Mc^{\infty,1}(\pi^\#)$ 
such that $\Op^\theta(a_0)^{-1}=\Op^\theta(b_0)$.
\end{enumerate}
\end{theorem}

\begin{proof}
Recall that the modulation spaces of symbols $M^{p,q}(\pi^\#)$ are independent on the choice of a window vector by \cite[Example 3.4(2)]{BB10c}. 
Then the assertions follow by the above Corollary~\ref{main_cor_mult} 
and Theorem~\ref{th_sj}. 
\end{proof}

\begin{theorem}\label{monot}
Assume that the representation $\pi$ is square integrable modulo the center of $G$ and let $\gg_e=\gg_e^1\dotplus\gg_e^2$ be any decomposition of the predual into a direct sum of linear subspaces. 
If $\phi\in\Hc_\infty$ and we have $1\le p_1\le p_2\le\infty$ and 
$1\le q_1\le q_2\le\infty$, then  
$M_\phi^{p_1,q_1}(\pi)\subseteq M_\phi^{p_2,q_2}(\pi)$. 
\end{theorem}

\begin{proof}
It follows by \eqref{not_eq1} that the proof will be complete as soon as we have proved that if $p,q\in[1,\infty]$ and $f\in M_\phi^{p,q}(\pi)$, 
then $f\in M_\phi^{\infty,\infty}(\pi)$, that is, 
$\Ac_\phi f\in L^\infty(\gg_e)$. 

In fact, let us define 
$$R_\phi\colon\gg_e\times\gg_e\to\CC,\quad 
R_\phi(X,Y)=(\pi(\exp_G X)\phi\mid \pi(\exp_G Y)\phi)
=(\Ac_\phi(\pi(\exp_G \phi)))(Y).$$
Now let us denote by $\ast_e$ the Baker-Campbell-Hausdorff multiplication on the nilpotent Lie algebra $\gg_e\simeq \gg/\zg$. 
There exists a polynomial map $\alpha\colon \gg_e\times\gg_e\to\RR$ 
such that 
$\pi(\exp_G ((-X)\ast Y))=\ee^{\ie\alpha(-X,Y)}\pi(\exp_G ((-X)\ast_e Y))$ 
(see for instance \cite{Ma07}), 
hence 
$$\begin{aligned}
R_\phi(X,Y)
&=\ee^{-\ie\alpha(-X,Y)}(\phi\mid \pi(\exp_G ((-X)\ast_e Y))\phi) \\
&=\ee^{-\ie\alpha(-X,Y)}(\Ac_\phi\phi)((-X)\ast_e Y).
\end{aligned}$$
Since $\Ac_\phi\phi\in\Sc(\gg_e)$ (see \cite{Pe94}) and the Lebesgue measure on $\gg_e$ coincides with the Haar measure on the nilpotent Lie group $(\gg_e,\ast_e)$, 
it then follows that 
\begin{equation}\label{monot_proof_eq1}
(\forall r,s\in[1,\infty])\quad 
\sup\limits_{X\in\gg_e}\Vert R(X,\cdot)\Vert_{L^{r,s}(\gg_e^1\times\gg_e^2)}<\infty.
\end{equation}
On the other hand, 
note that $R_\phi(X,Y)=(\Ac_\phi(\pi(\exp_G \phi)))(Y)$, hence 
\allowdisplaybreaks
\begin{align}
(\Ac_\phi f)(X)
&=(f\mid\pi(\exp_G X)\phi) \nonumber\\
&=(f\mid 
 \int\limits_{\gg_e}(\Ac_\phi(\pi(\exp_G \phi)))(Y)\pi(\exp_G Y)\phi\de Y) 
  \nonumber\\ 
&=(f\mid 
 \int\limits_{\gg_e}R_\phi(X,Y)\pi(\exp_G Y)\phi\de Y) 
  \nonumber\\
  &= 
 \int\limits_{\gg_e}\overline{R_\phi(X,Y)}(f\mid\pi(\exp_G Y)\phi) \de Y
  \nonumber
\end{align}
whence 
\begin{equation}\label{monot_proof_eq2}
(\forall X\in\gg_e)\quad (\Ac_\phi f)(X)=(\Ac_\phi f\mid R(X,\cdot)), 
\end{equation}
where the right-hand side makes sense since 
$\Ac_\phi f\in\Cc(\gg_e)\cap\Sc'(\gg_e)$ by \cite[Cor.~2.9(1)]{BB10c}, 
while $R(X,\cdot)\in\Sc(\gg_e)$ by \cite{Pe94}.  
If $f\in M_\phi^{p,q}(\pi)$, then it follows by~\eqref{monot_proof_eq2} 
along with H\"older's inequality in mixed-norm spaces (see \cite{BP61})
and~\eqref{monot_proof_eq1} that 
$$\sup\limits_{X\in\gg_e}\vert(\Ac_\phi f)(X)\vert 
\le 
\sup\limits_{X\in\gg_e} (\Vert \Ac_\phi f\Vert_{L^{p,q}(\gg_e^1\times\gg_e^2)}
\Vert R(X,\cdot)\Vert_{L^{p',q'}(\gg_e^1\times\gg_e^2)})<\infty,  $$
where $\frac{1}{p}+\frac{1}{p'}=\frac{1}{q}+\frac{1}{q'}=1$. 
Thus $\Ac_\phi f\in L^\infty(\gg_e)$, and this completes the proof, 
in view of the beginning remark. 
\end{proof}

\begin{corollary}\label{monot_cor}
If the representation $\pi$ is square integrable modulo the center of $G$, 
then for every $p\in[1,\infty]$ the modulation space $M^{p,1}(\pi^\#)$ is a subalgebra of $M^{\infty,1}(\pi^\#)$ endowed with the Moyal product~$\#$.  
\end{corollary}

\begin{proof}
As noted in \cite[Example 3.4(2) and Remark~3.7]{BB10c}, 
the representation  $\pi^\#\colon G\ltimes G\to\Bc(L^2(\Oc))$ is square integrable 
(modulo the center) and its modulation spaces are independent on the choice of the window vector. 
Thus the above Theorem~\ref{monot} applies for the representation $\pi^\#$ instead of $\pi$, and it follows that 
$M^{p_1,q_1}(\pi^\#)\subseteq M^{p_2,q_2}(\pi^\#)$ 
whenever $1\le p_1\le p_2\le\infty$ and $1\le q_1\le q_2\le\infty$. 

In particular we have $M^{p,1}(\pi^\#)\subseteq M^{\infty,1}(\pi^\#)$ 
if $1\le p\le\infty$. 
Moreover, it follows by Corollary~\ref{main_cor_mult} that 
if $a_1,a_2\in M^{p,1}(\pi^\#)$, then 
$a_1\# a_2\in M^{\frac{p}{2},1}(\pi^\#)\subseteq M^{p,1}(\pi^\#)$, 
and this completes the proof. 
\end{proof}

In the special case when $\pi$ is the Schr\"odinger representation of the Heisenberg group, the above result goes back to \cite{To01}; 
see also \cite{HTW07}. 
We also note that in this case we have $\Mc^{\infty,1}(\pi^\#)=M^{\infty,1}(\pi^\#)$.

\subsection*{Acknowledgment} 
The second-named author acknowledges partial financial support
from the Project MTM2007-61446, DGI-FEDER, of the MCYT, Spain, 
and from the grant PNII - Programme ``Idei'' (code 1194).

\end{document}